\documentclass[11pt]{article}
\usepackage{amsmath,amsfonts,graphicx,amsthm,float,amssymb,graphicx}
\usepackage[numbers,sort&compress]{natbib}
\usepackage[colorlinks=true,citecolor=black,linkcolor=black,urlcolor=blue]{hyperref}
\usepackage[noabbrev,capitalise]{cleveref}
\usepackage[lmargin=30mm,rmargin=30mm,tmargin=25mm,bmargin=25mm]{geometry}
\renewcommand{\baselinestretch}{1.1}
\renewcommand{\geq}{\geqslant}
\renewcommand{\leq}{\leqslant}
\allowdisplaybreaks
\newcommand{\N}{\mathbb{N}}
\newcommand{\half}{\ensuremath{\protect\tfrac{1}{2}}}
\newcommand{\quarter}{\ensuremath{\protect\tfrac{1}{4}}}
\newcommand{\third}{\ensuremath{\protect\tfrac{1}{3}}}
\newcommand{\FLOOR}[1]{\ensuremath{\protect\left\lfloor#1\right\rfloor}}
\usepackage[colorlinks=true,citecolor=black,linkcolor=black]{hyperref}
\newcommand{\doi}[1]{\href{http://dx.doi.org/#1}{\texttt{doi:#1}}}
\newcommand{\arXiv}[1]{\href{http://arxiv.org/abs/#1}{\texttt{arXiv:#1}}}
\newcommand{\urlprefix}{}
\theoremstyle{plain}
\newtheorem{theorem}{Theorem}
\newtheorem{lemma}[theorem]{Lemma}

\begin{document}

\title{{\bf The exact chromatic number of \\the convex segment    disjointness graph}\footnote{A preliminary version of this paper, which proved the lower bound in \cref{main}, was presented at the XIV Spanish Meeting on Computational Geometry (EGC 2011) and was published in the associated Hurtado Festschrift,   \emph{Lecture Notes in Computer Science} 7579:79--84, Springer, 2012. }\\[2ex]
{\normalsize\it In memory of Ferran Hurtado}}
  
\author{
Ruy Fabila-Monroy\footnote{Departamento de Matem\'aticas, Centro de Investigaci\'on y Estudios Avanzados del Instituto Polit\'ecnico Nacional, M\'exico, D.F., M\'exico    (\texttt{ruyfabila@math.cinvestav.edu.mx}).}  
    \and 
    Jakob   Jonsson\footnote{ \texttt{njakobj@gmail.com}. Research supported by the Swedish Research Council (grant 2006-3279).}  
    \and Pavel  Valtr\footnote{Department of Applied Mathematics, Charles University, Prague, Czech Republic (\texttt{valtr@kam.mff.cuni.cz}).}  
    \and 
    David R. Wood\footnote{School of Mathematical Sciences, Monash  University, Melbourne, Australia   (\texttt{david.wood@monash.edu}). }
    }

    
\maketitle

\begin{abstract}
Let $P$ be a set of $n$ points in strictly convex position in the plane.  Let $D_n$ be the graph whose vertex set is the set of all
line segments with endpoints in $P$, where disjoint segments are adjacent.  The chromatic number of this graph was first studied by
Araujo,  Dumitrescu, Hurtado, Noy, and Urrutia~[2005] and then by Dujmovi\'c and Wood~[2007].  Improving on their estimates, we prove the following exact formula:
$$\chi(D_n) = n - \left\lfloor \sqrt{2n +    \tfrac{1}{4}} - \tfrac{1}{2}\right\rfloor.$$
\end{abstract}

\section{Introduction}

Throughout this paper, $P$ is a set of $n$ points in strictly convex position in the plane.  The \emph{convex segment disjointness graph}, denoted by $D_n$, is the graph whose vertex set is the set of all line segments with endpoints in $P$, where two vertices are adjacent if the corresponding segments are disjoint. Obviously $D_n$ does not depend on the choice of $P$. Now assume that $P$ consists of $n$ evenly spaced points on a unit circle in the plane.  The graph $D_n$ was introduced by Araujo, Dumitrescu, Hurtado, Noy and Urrutia~\cite{ADHNU05}, who proved the following bounds on the
chromatic number of $D_n$:
\begin{equation*}
  \label{eqn:ADHNU}
  2\FLOOR{\third(n+1)}-1\leq \chi(D_n) <n-\half\FLOOR{\log n}\enspace.
\end{equation*}
Both bounds were improved by Dujmovi\'c and Wood~\cite{Antithickness} to
\begin{equation*}
  \tfrac{3}{4}(n-2)\leq \chi(D_n) 
  <n-\sqrt{\half n}-\half(\ln n)+4\enspace.
\end{equation*}
In this paper we prove matching upper and lower bounds, thus concluding the following exact formula for $\chi(D_n)$.
\begin{theorem}
\label{main}
$$\chi(D_n)=n-\FLOOR{\sqrt{2n+\quarter}-\half}.$$
\end{theorem}

Equivalently, $\chi(D_n) = n - k$, where $k$ is the unique integer satisfying $\binom{k+1}{2} \leq n < \binom{k+2}{2}$. 

\cref{main} is trivial for $n\leq 2$, so we henceforth assume that $n\geq 3$. The proof of the lower bound in \cref{main} is based on the observation that each colour class in a colouring of $D_n$ is a convex thrackle. We then prove that two maximal convex thrackles must share an edge in common. From this we prove a tight upper bound on the number of edges in the union of $k$ convex thrackles. \cref{main} quickly follows. These results are presented in \cref{LowerBound}. The proof of the upper bound in \cref{main} is given by an explicit colouring, which we describe in \cref{UpperBound}.

\section{Proof of Lower Bound}
\label{LowerBound}

A \emph{convex thrackle} on $P$ is a geometric graph with vertex set
$P$ such that every pair of edges intersect; that is, they have a
common endpoint or they cross. Observe that a geometric graph $H$ on
$P$ is a convex thrackle if and only if $E(H)$ forms an independent
set in $D_n$. A convex thrackle is \emph{maximal} if it is
edge-maximal.  As illustrated in \cref{fig:thrackle}, it is
well known and easily proved that every maximal convex thrackle $T$
consists of an odd cycle $C(T)$ together with some degree $1$ vertices
adjacent to vertices of $C(T)$. 
For each vertex $v$ in $C(T)$, let $W_T(v)$ be the convex wedge with apex $v$, such that the boundary
rays of $W_T(v)$ contain the neighbours of $v$ in $C(T)$. Then every
degree-1 vertex $u$ of $T$ lies in a unique wedge and the apex of this
wedge is the only neighbour of $u$ in $T$; see \citep[Lemma~1]{CN10} for a strengthening of these observations. 
See \cite{CN-DCG00,RST16,AS17,PP08,Cottingham93,GX17,PRT12,CKN15,FP11,PS11,GR95,PRS-DM94,CN09,FS35,HopfPann34,LPS-DCG97,Woodall-Thrackles,CMN04} for more on thrackles in general. 
Note that it is immediate from the above observations that every convex thrackle $T$ satisfies $|E(T)|\leq |V(T)|$. 
Conways's famous thrackle conjecture says this property holds for all thrackles. 
Note that $C(T)$ is an example of a \emph{musquash} \citep{CK-DM01,GN18}. 

\begin{figure}[htb]
  \begin{center}
    \includegraphics{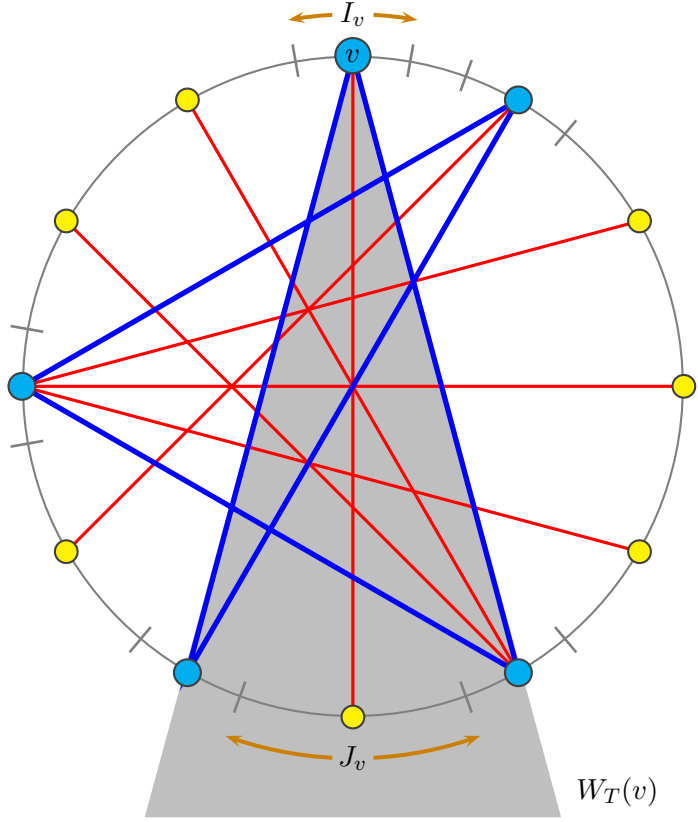}
  \end{center}
  \caption{A maximal convex thrackle $T$ with cycle $C(T)$ shown in blue. \label{fig:thrackle}}
\end{figure}


The following lemma is the heart of the proof of the lower bound in \cref{main}. We therefore include two proofs. 

\begin{lemma}
\label{common_edge}
  Let $T_1$ and $T_2$ be maximal convex thrackles on $P$.  
  Let $C_1 := V(C(T_1))$ and $C_2 := V(C(T_2))$. 
  Assume that $C_1\cap C_2=\emptyset$.  
  Then there is an edge in $T_1\cap T_2$, with one endpoint in $C_1$ and one endpoint in $C_2$.
\end{lemma}

\begin{proof}[Combinatorial Proof of \cref{common_edge}] 
Define a directed bipartite multigraph $H$ with bipartition $\{C_1,C_2\}$ as follows. 
For each vertex $u\in C_1$, add a \emph{blue} arc $uv$ to $H$, where $v$ is the unique vertex in $C_2$ for which $u\in W_{T_1}(v)$. 
Similarly,  for each vertex $u\in C_2$, add a \emph{red} arc $uv$ to $H$, where $v$ is the unique vertex in $C_1$ for which $u\in W_{T_2}(v)$. 
Since $C_1\cap C_2=\emptyset$, every vertex of $H$ has outdegree $1$. 
Thus $H$ contains a directed  cycle $\Gamma$. 
By construction, vertices in $H$ are not incident to an incoming and an outgoing edge of the same colour. Thus $\Gamma$  alternates between blue and red arcs. The red edges of $\Gamma$ form  a matching as well as the blue edges, both of which are  thrackles on the same set of points (namely, $V(\Gamma)$). However, there is only one matching thrackle on a set of points in convex position.  Therefore $\Gamma$ is a 2-cycle, which corresponds to an edge in $T_1\cap T_2$, with one endpoint in $C_1$ and one endpoint in $C_2$.
\end{proof}


Our second proof of \cref{common_edge} depends on the following topological notions. Let $S^1$ be the unit circle. 
For points $x,y\in S^1$, let $\overrightarrow{xy}$ be the clockwise  arc from $x$ to $y$ in $S^1$.  
A \emph{$\mathbb{Z}_2$-action} on $S^1$ is a homeomorphism $f:S^1\rightarrow S^1$ such that $f(f(x))=x$ for all
$x\in S^1$. Say that $f$ is \emph{free} if $f(x)\neq x$ for all $x\in S^1$.  

\begin{lemma}  
\label{fg}
If $f$ and $g$ are free $\mathbb{Z}_2$-actions of $S^1$, then $f(x)=g(x)$ for some point $x \in S^1$. 
  \end{lemma}
  
  \begin{proof}
Let $x_0\in S^1$. If $f(x_0)=g(x_0)$
  then we are done. Now assume that $f(x_0)\neq g(x_0)$. Without loss
  of generality, $x_0,g(x_0),f(x_0)$ appear in this clockwise order
  around $S^1$.  Parameterise $\overrightarrow{x_0g(x_0)}$ with a
  continuous injective function $p:[0,1]\rightarrow
  \overrightarrow{x_0g(x_0)}$, such that $p(0)=x_0$ and
  $p(1)=g(x_0)$. Assume that $g(p(t))\neq f(p(t))$ for all
  $t\in[0,1]$, otherwise we are done. Since $g$ is free, $p(t)\neq
  g(p(t))$ for all $t\in[0,1]$. Thus
  $g(p([0,1]))=\overrightarrow{g(p(0))g(p(1))}=\overrightarrow{g(x_0)x_0}$.
  Also $f(p([0,1]))=\overrightarrow{f(x_0)f(p(1))}$, as otherwise
  $g(p(t))= f(p(t))$ for some $t\in[0,1]$.  This implies that
  $p(t),g(p(t)),f(p(t))$ appear in this clockwise order around
  $S^1$. In particular, with $t=1$, we have $f(p(1))\in
  \overrightarrow{x_0g(x_0)}$. Thus
  $x_0\in\overrightarrow{f(x_0)f(p(1))}$.  Hence $x_0=f(p(t))$ for
  some $t\in[0,1]$.  Since $f$ is a $\mathbb{Z}_2$-action,
  $f(x_0)=p(t)$. This is a contradiction since
  $p(t)\in\overrightarrow{x_0g(x_0)}$ but
  $f(x_0)\not\in\overrightarrow{x_0g(x_0)}$.
  \end{proof}

\begin{proof}[Topological Proof of \cref{common_edge}] 
Assume that $P$ lies on $S^1$.  Let $T$ be a maximal convex thrackle
on $P$.  As illustrated in \cref{fig:thrackle}, for each
vertex $u$ in $C(T)$, let $(I_u,J_u)$ be a pair of closed intervals of
$S^1$ defined as follows. Interval $I_u$ contains $u$ and bounded by
the points of $S^1$ that are $\frac13$ of the way towards the first points
of $P$ in the clockwise and anticlockwise direction from $u$.  Let $v$
and $w$ be the neighbours of $u$ in $C(T)$, so that $v$ is before $w$
in the clockwise direction from $u$. Let $p$ be the endpoint of $I_v$
in the clockwise direction from $v$. Let $q$ be the endpoint of $I_w$
in the anticlockwise direction from $w$. Then $J_u$ is the interval
bounded by $p$ and $q$ and not containing $u$.  Define $f_T:S^1
\longrightarrow S^1$ as follows. For each $v\in C(T)$, map the
anticlockwise endpoint of $I_v$ to the anticlockwise endpoint of
$J_v$, map the clockwise endpoint of $I_v$ to the clockwise endpoint
of $J_v$, and extend $f_T$ linearly for the interior points of $I_v$
and $J_v$, such that $f_T(I_v)=J_v$ and $f_T(J_v)=I_v$. Since the
intervals $I_v$ and $J_v$ are disjoint, $f_T$ is a free
$\mathbb{Z}_2$-action of $S^1$.

By \cref{fg}, there exists $x \in S^1$ such that
  $f_{T_1}(x)=y=f_{T_2}(x)$.  Let $u\in C_1$ and $v \in C_2$ so
  that $x \in I_u \cup J_u$ and $x \in I_v \cup J_v$, where
  $(I_u,J_u)$ and $(I_v,J_v)$ are defined with respect to $T_1$ and
  $T_2$ respectively. Since $C_1\cap C_2=\emptyset$, we have
  $u\neq v$ and $I_u\cap I_v=\emptyset$. Thus $x\not\in I_u\cap
  I_v$. If $x\in J_u\cap J_v$ then $y\in I_u\cap I_v$, implying
  $u=v$. Thus $x\not\in J_u\cap J_v$.  Hence $x\in(I_u\cap
  J_v)\cup(J_u\cap I_v)$.  Without loss of generality, $x\in I_u\cap
  J_v$. Thus $y\in J_u\cap I_v$. If $I_u\cap J_v=\{x\}$ then $x$ is an
  endpoint of both $I_u$ and $J_v$, implying $u\in C_2$, which is a
  contradiction. Thus $I_u\cap J_v$ contains points other than $x$. It
  follows that $I_u\subset J_v$ and $I_v\subset J_u$.  Therefore the
  edge $uv$ is in both $T_1$ and $T_2$. Moreover one endpoint of $uv$
  is in $C_1$ and one endpoint is in $C_2$.
\end{proof}

\begin{theorem} 
\label{overcount} 
For every set $P$ of $n$ points  in strictly convex position, the union of $k$ maximal convex
  thrackles on $P$ has at most $k n-\binom{k}{2}$ edges.
\end{theorem}

\begin{proof}
  For a set $\mathcal{T}=\{T_1,\dots,T_k\}$ of $k$ maximal convex thrackles on $P$,
  define $C_i:=V(C(T_i))$ for $i\in[1,k]$, and let
$r(\mathcal{T})$ be the set of triples $(v,i,j)$ where $v \in C_i \cap C_j$ and $1\leq i<j\leq k$. 
The proof proceeds by induction on $|r(\mathcal{T})|$.

First suppose that $r(\mathcal{T})=\emptyset$.  Thus $C_i \cap C_j=\emptyset$ for all distinct $T_i,T_j\in \mathcal{T}$.  By \cref{common_edge}, $T_i$ and $T_j$ have an edge in common, with one endpoint in $C_i$ and one endpoint in $C_j$.  Hence distinct pairs of thrackles have distinct edges in common.  Since every maximal convex thrackle has $n$ edges and we overcount at least one edge for every pair, the total number of edges is at most $kn-\binom{k}{2}$.

Now assume that $r(\mathcal{T}) \neq\emptyset$. Thus there is a vertex $v$ and a pair of distinct thrackles $T_i$ and $T_j$, such that $v \in C_i \cap C_j$. We now modify $\mathcal{T}$ to create a new set $\mathcal{T'}$ of $k$  convex thrackles, as illustrated in \cref{split}. First, replace $v$ by two consecutive vertices $v'$ and $v''$ on $P$. Then, for each cycle $C_\ell$ with $v\in C_\ell$ and $\ell\neq j$ (which includes $C_i$), replace $v$ by $v'$ in $T_\ell$, and add the edge $xv''$ to $T_\ell$, where $x$ is the vertex in $C_\ell$ for which $v''$ is inserted into $W_{T_\ell}(x)$. Now, replace $v$ by $v''$ in $T_j$, and add the edge $yv'$ to $T_j$, where $y$ is the vertex in $C_j$ for which $v'$ is inserted into $W_{T_\ell}(y)$. Finally, for each cycle $C_a$ with $v\not\in C_a$, 
if $z$ is the vertex in $C_a$ with $v\in W_{T_a}(z)$, then replace the edge $zv$ by $zv'$ and $zv''$ in $T_a$. 
Let $\mathcal{T}'$ be the resulting set of thrackles.  Then $(v,i,j)\not\in r(\mathcal{T}')$, and every element of $r(\mathcal{T}')$ arises from an element of $r(\mathcal{T})$ (replacing $v$ by $v'$ or $v''$, as appropriate). Thus $r(\mathcal{T'}) \leq r(\mathcal{T})-1$. 
Since one edge is added to each thrackle, the number of edges in $\mathcal{T'}$ equals the number of edges in $\mathcal{T}$ plus $k$. By induction, $\mathcal{T}'$ has at most $k (n+1)-\binom{k}{2}$ edges, implying $\mathcal{T}$ has at most $k n -\binom{k}{2}$ edges.
\end{proof}

\begin{figure}
  \begin{center}
    \includegraphics{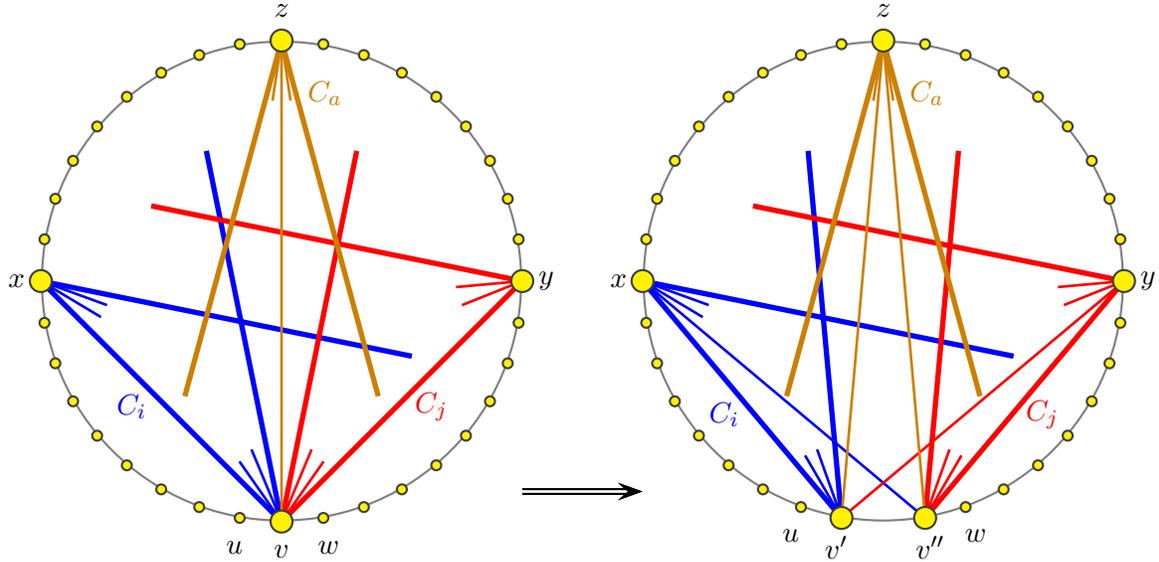}
  \end{center}
  \caption{\label{split} Construction in the proof of \cref{overcount}.}
\end{figure}

In the language of \citet{Antithickness}, \cref{overcount}
says that every $n$-vertex graph with convex antithickness $k$ has at
most $kn-\binom{k}{2}$ edges.

We now show that \cref{overcount} is best possible for all
$n\geq 2k$. Let $S$ be a set of $k$ vertices in $P$ with no two
consecutive vertices in $S$. If $v\in S$ and $x,v,y$ are consecutive
in this order in $P$, then $T_v:=\{vw:w\in
P\setminus\{v\})\}\cup\{xy\}$ is a maximal convex thrackle, and
$\{T_v:v\in S\}$ has exactly $k n-\binom{k}{2}$ edges in total.

\begin{proof}[Proof of Lower Bound in \cref{main}]
  If $\chi(D_n)=k$ then, there are $k$ convex thrackles whose union is
  the complete geometric graph on $P$.  Possibly add edges to obtain
  $k$ maximal convex thrackles with $\binom{n}{2}$ edges in total.  By
  \cref{overcount}, $\binom{n}{2}\leq kn-\binom{k}{2}$.
  The quadratic formula implies the result.
\end{proof}

\section{Proof of Upper Bound}
\label{UpperBound}

Label the points of $P$ by $1,2,\dots,n$ in clockwise order. Denote by $ab$ the line segment between points $a,b\in P$ with $a<b$, which is a vertex of $D_n$. 
%
It will be  convenient to adopt the matrix convention for indexing rows and columns in $\mathbb{Z}^2$. That is, row $a$ is immediately below row $a-1$, column $b$ is immediately to the right of column $b-1$, and $(a,b)$ refers to the lattice point in row $a$ and column $b$. Identify the vertex $ab$ of $D_n$ with the lattice point $(a,b)$ where $a<b$, which we represent as a unit square in our figures. Define $\Omega_n = \{ (i,j)\in\mathbb{Z}^2 : 1 \leq i < j \leq n\}$. We may consider $V(D_n)=\Omega_n$ represented as a triangle-shaped polyomino as illustrated in \cref{fig:10}(a).

\begin{figure}[htb]
\centering
(a) \includegraphics{size10path} 
\hspace*{10mm}
(b)  \includegraphics{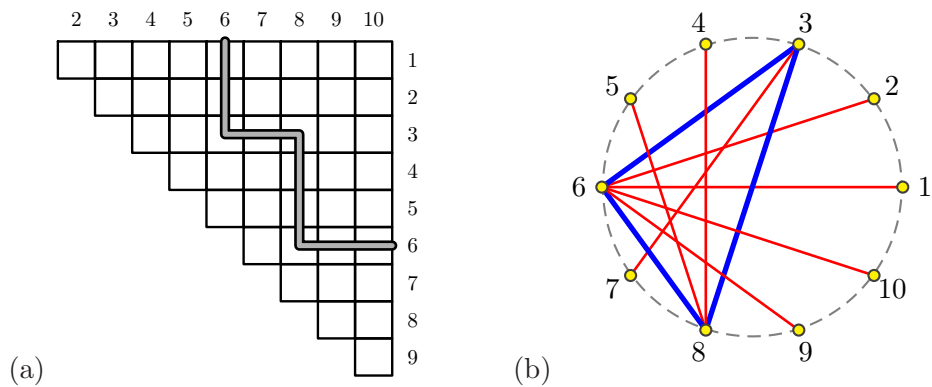}
\caption{(a) A maximal independent set in $D_{10}$ represented as a path in the polyomino $\Omega_{10}$. (b) The corresponding maximal convex thrackle $T$.  Turning points in the path correspond to vertices in $C(T)$. 
\label{fig:10}}
\end{figure}


Now, two distinct vertices $(a,b)$ and $(c,d)$ in $D_n$ are adjacent if and only if $a \leq c \leq b \leq d$ or $c \leq a \leq d \leq b$. In particular, for $(a,b)$ and $(c,d)$ to be non-adjacent, $(c,d)$ must lie in the nonshaded region in \cref{fig:17}. In particular, $(c,d)$ cannot be strictly southwest or strictly northeast of $(a,b)$. Moreover, $\max \{a,c\} \leq \min\{b,d\}$.

\begin{figure}[htb]
  \centering
    \includegraphics{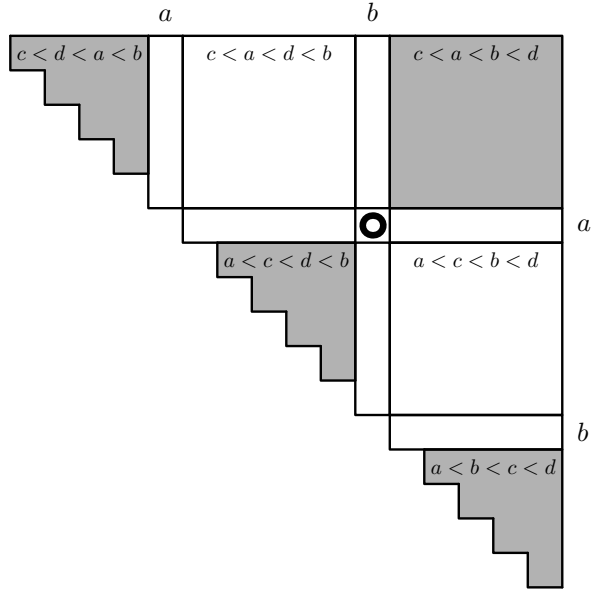}
  \caption{An element $(c,d)$ is adjacent to $(a,b)$ (marked with a
    thick circle) in the graph $D_n$ if and only if $(c,d)$ belongs to
    one of the shaded regions.
  \label{fig:17}}
\end{figure}

We conclude that every independent set $S$ of $D_n$ is a subset of some rectangle of the form $[1,r] \times [r,n]$ (with the southwest corner $rr$ removed). Namely, choose $(a,b), (c,d) \in S$ such that $a$ is maximal and $d$ is minimal.  Then $a' \leq a \leq d \leq b'$ for each $(a',b') \in S$.  In fact, it is straightforward to show that each maximal independent set forms a path from $(1,r)$ to $(r,n)$ for some $r \in\{2, \ldots, n-1\}$, where each step in the path is of the form $(i,j) \rightarrow (i,j+1)$ or $(i,j) \rightarrow (i+1,j)$.  An example is given in \cref{fig:10}(a).  Conversely, every such path is a maximal independent set.  We refer to such a path as a \emph{maximal thrackle path}; the corresponding set of line segments forms a maximal convex thrackle, as shown in \cref{fig:10}(b). 

\begin{figure}[!htb]
  \centering
    \includegraphics{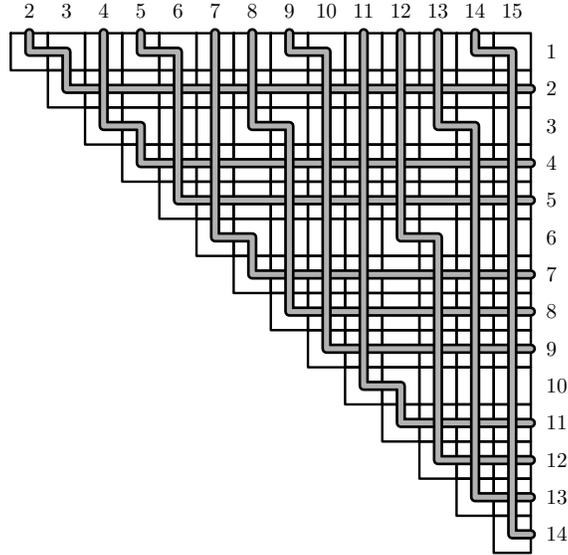}
  \caption{Ten thrackle paths covering $\Omega_{15}$.  
  \label{fig:15}
  }
\end{figure}

To summarize, the chromatic number of $D_n$ equals the minimum number of maximal thrackle paths that cover $\Omega_n$.  For example, \cref{fig:15} shows that it is possible to cover 
$\Omega_{15}$ with ten thrackle paths. As a consequence, $\chi(D_{15}) \leq 10$. Indeed, we have equality by the lower bound in \cref{main}.

For $k\ge1$, define the following intervals: 
$$\N_k:=[\tbinom{k}{2} + 1 , \tbinom{k+1}{2}] \quad\text{and}\quad \N'_k:=[ \tbinom{k}{2}+1 ,\tbinom{k+1}{2}-1].$$
Thus, $\N_1=\{1\}$, $\N_2=\{2,3\}$, $\N_3=\{4,5,6\}$, etc.
The sets $\N_k$ form a partition of $\N$. Observe that $|\N_k|=k$ and $|\N'_k|=k-1$ for each $k\geq 1$.

We now describe an infinite sequence of infinite paths covering
the infinite polyomino  $\Omega = \{ (i,j)\in\mathbb{Z}^2 : 1 \leq i < j \}$. 
The final construction for $\Omega_n$ is then 
obtained as a restriction of the covering to the set $\Omega_n$.
For each $k\ge2$ and for each $i\in \N'_k$, let $P_i$ be the following path: start at 
 $(1,i)$, walk south to $$\left(\binom{\binom{k+1}{2}-i+1}{2},i\right),$$
make one step east to 
$$\left(\binom{\binom{k+1}{2}-i+1}{2},i+1\right),$$
then walk south to $(i,i+1)$, and finally walk east through all the points in the $i$-th row. 

We now show that for each $j>1$, the paths $P_1,\dots,P_j$ cover all the points
in the $j$-th column. Let $j\in \N_k$. If $j=\binom{k}{2}+1$ then the path
$P_j$ covers the $j$-th column. If $j=\binom{k+1}{2}$ then the 
path $P_{j-1}$ covers the $j$-th column. Now assume that $j\ne\binom{k}{2}+1$
and $j\ne\binom{k+1}{2}$. Let $\ell:=\binom{k+1}{2}-j$. The path $P_j$ covers the topmost
$\binom{\ell+1}{2}$ points in the $j$-th column. The next
$\ell$ points of the $j$-th column lie in the rows 
$\binom{\ell+1}{2}+1,\dots,\binom{\ell+2}{2}-1$. These rows are completely covered
by the $\ell$ paths $P_h$ where $h\in\N'_{\ell+1}$. 
The remaining bottom part of the $j$-th column from $(\binom{\ell+2}{2},j)$
to $(j-1,j)$ is covered by $P_{j-1}$.

Now consider the restriction of the paths $P_1,\dots,P_n$ to the triangular polyomino $\Omega_n$. Each intersection $P_i\cap\Omega_n$ is a maximal thrackle path in $\Omega_n$. Let $k$ be the unique integer satisfying $\binom{k+1}{2} \leq n < {k+2\choose2}$. Then the above construction gives a covering of the polyomino $\Omega_n$ by $n-k$ thrackle paths, since a path $P_i$ exists for each $i\leq n$, except for the $k$ values $i=\binom{2}{2},\binom{3}{2},\dots,\binom{k+1}{2}$. The upper bound in \cref{main} follows.

\def\soft#1{\leavevmode\setbox0=\hbox{h}\dimen7=\ht0\advance \dimen7
  by-1ex\relax\if t#1\relax\rlap{\raise.6\dimen7
  \hbox{\kern.3ex\char'47}}#1\relax\else\if T#1\relax
  \rlap{\raise.5\dimen7\hbox{\kern1.3ex\char'47}}#1\relax \else\if
  d#1\relax\rlap{\raise.5\dimen7\hbox{\kern.9ex \char'47}}#1\relax\else\if
  D#1\relax\rlap{\raise.5\dimen7 \hbox{\kern1.4ex\char'47}}#1\relax\else\if
  l#1\relax \rlap{\raise.5\dimen7\hbox{\kern.4ex\char'47}}#1\relax \else\if
  L#1\relax\rlap{\raise.5\dimen7\hbox{\kern.7ex
  \char'47}}#1\relax\else\message{accent \string\soft \space #1 not
  defined!}#1\relax\fi\fi\fi\fi\fi\fi}

\end{document}